\documentclass[12pt,reqno]{amsart}

\usepackage{amsmath,amssymb,amsfonts,amsthm,mathrsfs}
\usepackage{graphicx}
\usepackage[usenames,dvipsnames]{color}
\usepackage{enumerate}
\usepackage[margin=1in]{geometry}

\newtheorem{thm}{Theorem}%[section]
\newtheorem{lem}[thm]{Lemma}
\newtheorem{cor}[thm]{Corollary}

\newcommand{\prob}[1]{\mathrm{Pr}\left ( #1 \right )}
\newcommand{\E}[1]{\mathbb{E} \left [ #1 \right ]}
\newcommand{\nat}{{\mathbb N}}

\newcommand{\floor}[1]{\left\lfloor#1\right\rfloor}
\newcommand{\ceil}[1]{\left\lceil#1\right\rceil}
\newcommand{\cp}{\mathbin{\square}}
\DeclareMathOperator{\capt}{capt}

\DeclareMathOperator{\dist}{dist}
\DeclareMathOperator{\rad}{rad}
\DeclareMathOperator{\diam}{diam}
\DeclareMathOperator{\Bin}{Bin}

\title{The game of Overprescribed Cops and Robbers played on graphs}
\author{Anthony Bonato}

\address{Department of Mathematics\\
Ryerson University\\
Toronto, ON\\
Canada} \email{abonato@ryerson.ca}

\author{Xavier P\'erez-Gim\'enez}
\address{Department of Mathematics\\
University of Nebraska-Lincoln\\
Lincoln, NE\\
U.S.A.} \email{xperez@unl.edu}

\author{Pawe{\l} Pra{\l}at}
\address{Department of Mathematics\\
Ryerson University\\
Toronto, ON\\
Canada
\and
The Fields Institute for Research in Mathematical Sciences\\
Toronto, ON\\
Canada} \email{pralat@ryerson.ca}

\author{Benjamin Reiniger}
\address{Department of Applied Mathematics\\
Illinois Institute of Technology\\
Chicago, IL \\
U.S.A.} \email{breiniger@iit.edu}

\thanks{The first and third authors are supported by grants from NSERC}

\begin{document}

\begin{abstract}
We consider the effect on the length of the game of Cops and Robbers when more cops are added to the game play. In Overprescribed Cops and Robbers, as more cops are added, the capture time (the
minimum length of the game assuming optimal play) monotonically decreases. We give the full range of capture times for any number of cops on trees, and classify the capture time for an asymptotic
number of cops on grids, hypercubes, and binomial random graphs. The capture time of planar graphs with a number of cops at and far above the cop number is considered.
\end{abstract}

\maketitle

\section{Introduction}

The game of Cops and Robbers, first introduced in~\cite{AF,nw,q}, has attracted considerable recent interest among graph theorists. The game is played on a reflexive graph; that is, each vertex has
at least one loop. Multiple edges are allowed, but make no difference to the play of the game, so we always assume there is exactly one edge joining adjacent vertices. There are two players,
consisting of a set of \emph{cops} and a single \emph{robber}. The game is played over a sequence of discrete time-steps or turns, with the cops going first on turn 0 and then playing on alternate
time-steps. We refer to the set of cops as $C$ and the robber as $R$. When a player is ready to move in a round they must move to a neighbouring vertex. Because of the loops, players can pass, or
remain on their own vertices. Observe that any subset of $C$ may move in a given round. The cops win if after some finite number of rounds, one of them can occupy the same vertex as the robber (in a
reflexive graph, this is equivalent to the cop landing on the robber). This is called a \emph{capture}. The robber wins if he can evade capture indefinitely. A \emph{winning strategy for the cops} is
a set of rules that if followed, result in a win for the cops. A \emph{winning strategy for the robber} is defined analogously.

The \emph{cop number} of a graph, first introduced in \cite{AF}, is the minimum number of cops needed to have a winning strategy. The cop number is often a challenging graph parameter to analyze, and
establishing upper bounds for this parameter is the focus of Meyniel's conjecture: the cop number of a connected $n$-vertex graph is $O(\sqrt{n}).$

The \emph{length} of a game is the number of rounds it takes (not including the initial or $0$th round) to capture the robber. We say that a play of the game with $c(G)$ cops is \emph{optimal} if its
length is the minimum over all possible strategies for the cops, assuming the robber is trying to evade capture for as long as possible (here $c(G)$ denotes the cop number of $G$). If $k$ cops play
on a graph with $k \ge c(G)$, we denote this invariant $\capt_k(G)$, which we call the $k$-\emph{capture time} of $G$. In the case $k=c(G)$, we just write $\capt(G)$ and refer to this as the
\emph{capture time} of $G$. Note that $\capt_k(G)$ is trivially $0$ if $k\ge n$, where $n$ is the order of $G$; and $\capt_k(G)=1$ if $\gamma (G) \le k < n$, where $\gamma (G)$ denotes the domination
number of $G$. Hence, the analysis of this invariant can be restricted to the range $c(G) \le k \le \gamma(G)$. (We can assume that $\gamma (G)<n$ and thus, $\capt_{\gamma(G)}(G)=1$, by excluding the
degenerate case in which $G$ is a co-clique.) Observe that $\capt_k(G)$ is monotonically decreasing with $k$. We refer to this effect as \emph{temporal speed-up}.

The capture time was introduced in \cite{BGHK}. From \cite{BGHK,gav1} it was shown that if $G$ is cop-win (that is, has cop number $1$) of order $n\ge7,$ then $\capt(G) \leq n-4,$ and there are
planar cop-win graphs that prove that the bound of $n-4$ is optimal. Mehrabian~\cite{meh} investigated the capture time of Cartesian grids, and proved that if $G$ is the Cartesian product of two
trees, then $\capt(G)=\lfloor \diam(G)/2 \rfloor .$  In particular, the $2$-capture time of an $m \times n$ Cartesian grid is $\lfloor \frac{m+n}{2} \rfloor -1$.  The capture time of hypercubes was
studied in \cite{BGKP}, where the authors used the probabilistic method to prove that $\mathrm{capt}(Q_n) = \Theta(n \ln n)$. See also \cite{gen}.

In the present work, we consider \emph{Overprescribed Cops and Robbers} games, where the number of cops is \emph{strictly greater} than the cop number. We study temporal speed-up for various graph
classes, such as trees (see Section~\ref{sec:prelims}), grids and hypercubes (see Section~\ref{sec:gridsH}), and planar graphs (see Section~\ref{sec:planar}). For trees, we derive the precise value
of $\capt_k(G)$ for all $k$ using metric $k$-centers. We give the asymptotic order of all values of temporal speed-up for grids and hypercubes (with one small exception in the range of $k$ for
hypercubes). We analyze temporal speed-up for planar graphs playing with $\Omega(\sqrt{n})$ cops, and bounds on the $k$-capture time on planar graphs playing with $k=3$ (which is the upper bound for
the cop number of planar graphs). The paper finishes with a discussion of temporal speed-up in binomial random graphs.

We consider finite undirected, reflexive graphs. For additional background on Cops and Robbers and Meyniel's conjecture, see the book~\cite{BN}. For additional background on graph theory, see~\cite{west}.

\section{Trees and retracts}\label{sec:prelims}

The classification of the $k$-capture times for all $k$ for the class of trees is relatively straightforward. Hence, we begin with this class as a warm up. Along the way, we prove an elementary but
useful theorem relating $k$-capture time to retracts.

For an integer $k\geq1$ and a graph $G$, let $\rad_k(G)$ denote the \emph{$k$-center radius}, defined as
\[\min_{\substack{S\subseteq V(G)\\|S|\leq k}} \max_{v\in V(G)} \dist(v,S).\]
A set $S$ achieving the minimum is called a \emph{metric $k$-center} of $G$. Note that $\rad_1(G)$ is just the radius of
$G$; we will drop the subscript 1 in this case. Computing the $k$-center radius is \textbf{NP}-hard for general graphs; see \cite{vaz}.

A \emph{retract} of a graph $G$ is an induced subgraph $H$ for which there exists a graph homomorphism from $G$ to $H$ whose restriction to $H$ is the identity. Retracts play an important role in the
game of Cops and Robbers as noted in \cite{AF}, who proved that for a retract $H$ of $G$, $c(H)\leq c(G)$.  We note the following observation, likely part of folklore.

\begin{lem}\label{thm:kradiusretract}
For a retract $H$ of a graph $G$, $\rad_k(H)\leq\rad_k(G)$.
\end{lem}
\begin{proof}
Let $S$ be a metric $k$-center of $G$, let $r=\rad_k(G)$, and let $f$ be a homomorphism witnessing that $H$ is a retract of $G$. Consider the set $f(S)$ in $H$.  We claim that every vertex of $H$ is
within distance $r$ of some vertex of $f(S)$, from which it follows that $\rad_k(H)\leq r$.

Let $v$ be a vertex of $H$ and let $s\in S$ be such that $\dist_G(v,s)\leq r$.  Consider a $(v,s)$-walk $W$ of length at most $r$ in $G$. Then $f(W)$ is a $(v,f(s))$-walk of length at most $r$ in
$H$.
\end{proof}
Note that cop number and $k$-center radius are not monotonic under subgraphs or induced subgraphs: for instance, adding a universal vertex drops both parameters to 1.

Metric $k$-centers give us an elementary method to lower bound the $k$-capture time.

\begin{lem}\label{thm:kcenterlb}
For any graph $G$, $\capt_k(G)\geq \rad_k(G)$.
\end{lem}
\begin{proof}
We need to provide a strategy for the robber.  After the cops have chosen their starting vertices, we place the robber on a vertex of maximum distance from any cop and just keep her there throughout the game.
\end{proof}
The following corollary is immediate.

\begin{cor}
For any $G$ and any $k$, $\capt_k(G)\geq \frac{\diam(G)-k+1}{2k}$.
\end{cor}
\begin{proof}
Let $d=\diam(G)$.  Then $G$ has a copy of $P_{d+1}$ as a retract; see \cite{AF,BN}. (For a direct argument, let $P$ be a shortest path of length $d$ with end-vertices $x$ and $y$. Map each vertex of
$G$ to the vertex of $P$ which is the same distance to $x.$) By Lemma~\ref{thm:kradiusretract} we have that $\rad_k(G)\geq \rad_k(P_{d+1})$. The balls of radius $r$ in $P_{d+1}$ have size at most
$2r+1$, so if a set of $k$ balls are to cover the vertex set, we must have that $k(2r+1)\geq d+1$. Hence, $\rad_k(P_{d+1})\geq \frac{d+1}{2k}-\frac{1}{2}$.
\end{proof}

The following theorem establishes the capture time of trees.

\begin{thm}\label{thm:capt1tree}
For any tree $T$, $\capt(T)=\rad(T)$.
\end{thm}
\begin{proof}
The lower bound follows from Lemma~\ref{thm:kcenterlb}.
For the upper bound on $\capt(T)$, we give a strategy for the cop.  He initially places himself on a central vertex of $T$, and at each step moves along the unique path between himself and the
robber.  Rooting the tree at his starting vertex, this implies that the robber is always in the subtree rooted at the cop's current position (if not already caught), and so she is caught in at most
$\rad(T)$ steps.
\end{proof}

\smallskip

The next theorem is useful to bound the capture time of a graph when there are many more cops than are needed to capture the robber.  We will see in the later sections that it gives the correct
capture time up to a constant factor for grids and, in some cases, hypercubes.

\begin{thm}\label{thm:retractptn}
Suppose that $V(G)=V_1\cup\dotsb\cup V_t$, where $G[V_i]$ is a retract of $G$ for every $i$ and $k=\sum_{i\in[t]} k_i$.  Then $\capt_k(G)\leq\max_{i\in[t]} \capt_{k_i}(G[V_i])$. Note that if
$k_i<c(G[V_i])$, then we say that $\capt_{k_i}(G[V_i])=\infty$.
\end{thm}
\begin{proof}
We give a strategy for the cops.  For each $i$, we assign a team of $k_i$ cops to $G[V_i]$, which we refer to as the \emph{territory} of those cops. Each team of cops plays their optimal strategy on
their territory to capture the image of the robber under the retract to $G[V_i]$.  After $\max_i \capt_{k_i}(G[V_i])$ turns, every team of cops has caught their projection of the robber; in
particular, some team of cops has caught the robber.
\end{proof}
The following corollary gives the $k$-capture time for trees for all $k.$

\begin{cor}
For any tree $T$, $\capt_k(T)=\rad_k(T)$.
\end{cor}
\begin{proof}
The lower bound follows from Lemma~\ref{thm:kcenterlb}. For the upper bound, let $\{v_1, \dotsc, v_k\}$ be the vertices of metric $k$-center of $T$.  Take $V_i=B(v_i, \rad_k(T))$. By the definition
of $\rad_k(T)$, we have $V(T)=\bigcup_{i\in[k]} V_i$.  With $k_i=1$ for every $i$, Theorem~\ref{thm:retractptn} and Theorem~\ref{thm:capt1tree} imply the result.
\end{proof}

\section{Cartesian grids and hypercubes}\label{sec:gridsH}

Given $d, q \in\nat$, let $G^d_q = \square_{i=1}^d P_q$ be the $d$-dimensional Cartesian grid on $q^d$ vertices. For $q\geq2$, the cop number of $G^d_q$ is $\ceil{(d+1)/2}$~\cite{MM}, with associated
capture time less than $\frac12 qd\ceil{\log_2 d}$ (see Theorem~3 in~\cite{BGKP}). We will first consider the Cartesian grid $G^d_q$ of constant dimension $d$ with $q=n$ for some $n\to\infty$. Later
on, we will shift our attention to the case in which $q=2$ and the dimension $d=n$ for some $n\to\infty$. In that second case, the Cartesian grid $G^n_2$ is also known as the hypercube and we will
denote it by $Q_n$. All asymptotic notations in this section are with respect to $n$, as $n$ grows to infinity.

The following theorem gives the asymptotic order of the $k$-capture time of the $d$-dimensional Cartesian grid $G^d_n$, for constant $d$. Note that the domination number is trivially $\Theta(n^d)$.
\begin{thm}\label{thm:grids}
Fix any constant $d\in\nat$, and let $k=k(n)$ be such that $k\geq c(G_n^d)$.  If $k = O(n^d)$, then $\capt_k(G^d_n)=\Theta\left(n / k^{1/d} \right)$.
%If $k=O(n^d\log n)$, then a.a.s.\
%$\rcapt_k(G^d_n)=\Theta\left( n (\log k/k)^{1/d}\right)$.
\end{thm}

\begin{proof}
We will first prove the upper bound on $\capt_k(G^d_n)$. Cover the grid $G^d_n$ by $\floor{k/c(G_n^d)}$ subgrids isomorphic to $G^d_{n'}$ for $n'=\Theta(n/ k^{1/d})$. (The subgrids may overlap, but
they cover all vertices of $G^d_n$.)  Since subgrids are retracts of the whole grid, Theorem~\ref{thm:retractptn} with $k_i=c(G_n^d)=c(G_{n'}^d)=\ceil{(d+1)/2}$ for all $i$ and $t=\floor{k/c(G_n^d)}$
gives the bound $\frac12 n' d\ceil{\log_2 d} = O(n')$ (since $d$ is constant) as desired. To prove the lower bound on $\capt_k(G^d_n)$, pack $k+1$ pairwise-disjoint subgrids isomorphic to $G^d_{n''}$
for $n''=\Theta(n/ k^{1/d})$, and place the $k$ cops in any arbitrary way. By the pigeonhole principle, at least one subgrid contains no cop. The robber starts in that subgrid and survives for
$\Omega(n/k^{1/d})$ rounds by not moving.
\end{proof}

We now consider the temporal speed-up of hypercubes. The cop number of $Q_n$, the hypercube on $2^n$ vertices, is $\ceil{\frac{n+1}{2}}$, with the associated capture time $\Theta(n\log n)$. The
coefficient hidden in the $\Theta(\cdot)$ notation is between $1/2$ and 1; see~\cite{BGKP} for more details. On the other hand, the domination number of $Q_n$ is $(1+o(1))\frac{2^n}{n}$, with the
associated capture time 1. Our goal in this section is to investigate the capture time for the number of cops between the cop number and the domination number.

\subsection{Upper bounds}

Let us start with the following result that works well for a small number of cops. Let us mention that this bound is not needed to prove Corollary~\ref{thm:Qnsummary} that summarizes results for
hypercubes but concentrates only on the order of magnitude of the $k$-capture time. However, it does give better constants in certain ranges of $k$.

\begin{thm}\label{thm:Qn-upperbound1}
Let $\omega = \omega(n)$ be a function tending to infinity arbitrarily slowly. Suppose that $k=k(n)$ is such that $c(Q_n) \le k \le 2^n/\omega$. Then
\[
\capt_k(Q_n) \le(1+o(1)) \log_2(2^n/k)\log\log(2^n/k).
\]
In particular, if $k=b^n$ for some $1<b<2$, then
\[
\capt_k(Q_n) \le(1+o(1)) ( 1 - \log_2b ) n \log n.
\]
\end{thm}
\begin{proof}%[Proof (sketch)]
Given $\ell$ a nonnegative integer such that $\ell \le n$, we can partition $V(Q_n)$ into $2^{n-\ell}$ sets, each inducing a copy of $Q_\ell$. The cop number of $Q_\ell$ is $c=\ceil{(\ell+1)/2} \le
\ell$, and the capture time is at most $(1+o(1))\ell\log\ell$ (see~\cite{BGKP}). Since subcubes are retracts (the standard projection maps are homomorphisms for the \emph{reflexive} cubes), we apply
Theorem~\ref{thm:retractptn} with $t=2^{n-\ell}$ and $k_i\geq c$ for all $i$.
This requires the number of cops to be at least $2^{n-\ell} c$, which clearly holds if $2^\ell / \ell \ge 2^n / k$. In order to obtain the best bound, we choose $\ell$ to be the minimum integer such
that $2^\ell / \ell \ge 2^n / k$. Since $k\leq 2^n/\omega$, it follows that $2^n/k\to\infty$ and so $\ell \sim \log_2 ( 2^n / k)$.  (Indeed, note that if $\ell = \log_2 ( 2^n / k)$, then
$2^{\ell}/\ell = (2^n/k)/\ell < 2^n/k$; but if $\ell = \log_2 ( 2^n / k) + 2 \log_2 \log_2 ( 2^n / k) \sim \log_2 ( 2^n / k)$, then $2^{\ell}/\ell \sim (2^n/k) \ell \ge 2^n/k$.) The conclusion from
Theorem~\ref{thm:retractptn} is then $\capt_k(Q_n)\leq \capt_{k_i}(Q_{\ell})$, which is in turn at most $\capt_c(Q_{\ell})\sim \ell\log\ell \sim \log(2^n/k)\log\log(2^n/k)$.
\end{proof}

To prove the next result, we will use the following version of \emph{Chernoff's bound}. Suppose that $X \in \Bin(n,p)$ is a binomial random variable with expectation $\mu=np$. If $0<\delta<3/2$, then
\begin{equation}\label{eq:Chern}
\prob{ |X- \mu| \ge \delta \mu } \le 2 \exp \left( -\frac{\delta^2 \mu}{3} \right).
\end{equation}
(For example, see Corollary~2.3 in~\cite{JLR}.) It is also true that~(\ref{eq:Chern}) holds for a random variable with the hypergeometric distribution. The \emph{hypergeometric distribution} with
parameters $N$, $n$, and $m$ (assuming $\max\{n,m\} \le N$) is defined as follows. Let $\Gamma$ be a set of size $n$ taken uniformly at random from set $[N]$. The random variable $X$ counts the
number of elements of $\Gamma$ that belong to $[m]$; that is, $X = |\Gamma \cap [m]|$. It follows that~(\ref{eq:Chern}) holds for the hypergeometric distribution with parameters $N$, $n$, and $m$,
with expectation $\mu = nm/N$. (See, for example, Theorem~2.10 in \cite{JLR}.)

Given a vertex $v$ and an integer $0\le i\le n$, $N_i(v)$ is the set of vertices at distance exactly $i$ from $v$ and $N_{\le i}(v) = \bigcup_{j=0}^i N_j(v)$. Now, we are ready to state the upper
bound that works well for a large number of cops.

\begin{thm}\label{thm:Qn-upperbound2}
Suppose that $k = k(n) \in \nat$ is such that
\begin{equation}\label{eq:cops_lower_bound}
k\ge 36 \cdot 2^n \ \frac {(2d+1)_{d+1}}{ (n-d)_{d+1} },
\end{equation}
for some $d \le cn-2$, where $c = 1/2 - \sqrt{2}/4 \approx 0.1464$. Then, for $n$ large enough,
$$
\capt_k(Q_n) \le 2d+1.
$$
In particular, the desired upper bound for the capture time holds provided that
$$
k\ge 36 \cdot 2^n n \ \left( \frac {3d}{n(1-c)} \right)^{d+1}.
$$
\end{thm}

Before we move to the proof of the theorem, let us mention that the condition for $d$ is, in some sense, not needed and it does not make the result weaker. Indeed, note that after replacing $d$ by
$d+1$, the lower bound~(\ref{eq:cops_lower_bound}) for the number of cops is affected by the multiplicative constant
$$
\frac { (2d+3)_{d+2} / (n-d-1)_{d+2} } { (2d+1)_{d+1} / (n-d)_{d+1} } = \frac { (2d+3)(2d+2)(n-d) } { (d+1)(n-2d-1)(n-2d-2) } \le \frac { 4c (1-c)}{ (1-c)^2 } = 1.
$$
Hence, only up to this point the lower bound for the number of cops is a decreasing function of $d$. After removing this artificial restriction on $d$, there would be more choices for $d$ to satisfy
the desired condition but clearly one should consider the smallest value of $d$ to get the best bound.

We observe that since $Q_n$ is a relatively good expander, the proof follows similar ideas as the ones used to bound the cop number for random graphs~\cite{Expansion3,Expansion1, Expansion2}.

\begin{proof}[Proof of Theorem~\ref{thm:Qn-upperbound2}]
We distribute all the cops \emph{at random}; that is, cops select a set of vertices of cardinality $k$ uniformly at random, and then they start on this set.
 Suppose that the robber starts the game on vertex $v$. Our goal is to show that a.a.s., regardless where she starts, after $d+1$ (cops') moves the cops can completely occupy
$N_d=N_d(v)$. As the first move belongs to the cops, the robber will not be able to escape from the ball $N_{\le d}(v)$ around her initial position; she will be ``trapped'' there.

We are going to show that with probability at least $1-2^{-2n+2}$, there exists a matching saturating $N_d$ between vertices of $N_d$ and cops initially occupying $N_{2d+1}=N_{2d+1}(v)$. In order to
do it, we are going to use Hall's theorem for matchings in bipartite graphs. A ``neighbour'' in $N_{2d+1}$ of a vertex $w\in N_d$ (in this auxiliary bipartite graph) is a vertex in $N_{2d+1}$ that
contains a cop and is at distance exactly $d+1$ from $w$. For a given $S\subseteq N_d$ of size $s=|S|\ge1$, we wish to find $t_s$, a lower bound for the number of vertices in $N_{2d+1}$ at distance
$d+1$ from some vertex in $S$. As each vertex in $N_d$ has $\binom{n-d}{d+1}$ vertices in $N_{2d+1}$ that are at distance $d+1$, and each vertex in $N_{2d+1}$ has $\binom{2d+1}{d+1}$ vertices in
$N_d$ that are at distance $d+1$, we get
\[
t_s \ge \frac {\binom{n-d}{d+1}} {\binom{2d+1}{d+1}}  s = \frac{ (n-d)_{d+1} } { (2d+1)_{d+1} }  s.
\]
Let $X$ be the random variable counting how many of these vertices initially contain cops. Using the assumption for $k$, we get that $\E{X} \ge t_s k/2^{n} \ge 36 n s$, and it follows from Chernoff's
bound (applied to $X$, a hypergeometric random variable) that
\[
\prob{X < s} \le \prob{X \le \E{X}/2 } \le 2 \exp(-  \E{X}/12 ) \le 2 \exp(- 3 n s).
\]
Taking a union bound over all $\binom{\binom{n}{d}}{s} \le \binom{2^n}{s} \le 2^{ns}$ choices for sets $S$ of cardinality $s$, we conclude that with probability at least $1- 2^{-2ns+1}$, Hall's
condition holds for all sets of size $s$. Summing the failure probability over $1\le s\le\binom{n}{d}$, we get that the desired condition holds for all sets with probability at least $1-2^{-2n+2}$.
Finally, by taking a further union bound over all $2^n$ choices for $v$, the initial vertex the robber starts on, we conclude that a.a.s., regardless where the robber initially starts, the desired
matching can be found. We may assume then that this is the case.

Let us suppose that the robber starts at vertex $v$.  We give a strategy for the cops for the remainder of the game.  The cops in $N_{2d+1}(v)$ move to destinations in $N_d(v)$ according to the
matching guaranteed above (moving along any shortest path), thereby occupying every vertex of $N_d(v)$, taking exactly $d+1$ steps.  As we already mentioned, the robber is now ``trapped'' in the ball
around $v$. In the next $d$ steps, the cops move towards $v$ by covering at each step one full layer $N_i(v)$.  Note that for any  $i$ with $1\le i \le d-1$ (and in particular $i<cn<n/2$), there
exists a matching between $N_i(v)$ and $N_{i+1}(v)$ saturating $N_i(v)$. Indeed, arguing as before we notice that for any $S \subseteq N_i(v)$, $|N(S) \cap N_{i+1}(v)| \ge \frac {n-i}{i+1} |S| \ge
|S|$ and so Hall's condition holds for the bipartite graph induced by layers $N_i(v)$ and $N_{i+1}(v)$. The robber is captured after another $d$ steps, and the proof is finished.
\end{proof}

\subsection{Lower bounds}

As in the previous subsection, let us start with the results that works well for a small number of cops.

\begin{thm}\label{thm:Qn-lowerbound1}
Fix any constants $0<\alpha<\alpha'<1$, and suppose that $c(Q_n) \le k = k(n) \le e^{n^\alpha}$. Then,
\[
\capt_k(Q_n) \ge \frac{1-\alpha'}{2}(n-1)\log n.
\]
\end{thm}
\begin{proof} We provide a sketch of the proof only. The robber performs a random walk on $Q_n$. Following the proof of Theorem~8 in~\cite{BGKP} with $T=(1/2)(n-1)\log n$ and $\epsilon=\alpha'$, the probability that any of the cops captures the robber
in under $(1-\epsilon)T$ rounds is at most $k \exp(-(n/2)^{\alpha'}/4) = o(1)$.
\end{proof}

The next result works well for a large number of cops.

\begin{thm}\label{thm:Qn-lowerbound2}
Suppose that $k = k(n) \in \nat$ is such that
$$
k< \frac {2^n}{\sum_{i=0}^d \binom{n}{i}}.
$$
Then,
$$
\capt_k(Q_n) > d.
$$
In particular, the desired lower bound for the capture time holds, provided that
$$
k < \frac 12 \cdot 2^n \left( \frac {d}{en} \right)^{d},
$$
for some $d \le n/3$.
\end{thm}
\begin{proof}
In $d$ steps, any cop can reach $\sum_{i=0}^d\binom{n}{i}$ vertices. Therefore, regardless of how cops are initially distributed, they can reach at most $k \sum_{i=0}^d\binom{n}{i} < 2^n$ vertices in
$d$ steps. Hence, the robber can pick an initial vertex that is at distance at least $d+1$ from any cop and stay put. She clearly survives for more than $d$ rounds.

The second part follows from the fact that for any $d \le n/3$ we have
\[
\frac {2^n}{\sum_{i=0}^d \binom{n}{i}} \ge \frac {2^n}{2 \binom{n}{d}} \ge \frac {2^n}{2 (en/d)^d}. \qedhere
\]
\end{proof}

\subsection{Summary and open questions}

In this section we summarize the results for hypercubes, highlighting what remains to be investigated.
It seems that the behaviour of the capture time is well understood for all cases except part~(ii).

\begin{cor}\label{thm:Qnsummary}
Let $\varepsilon >0$,
\begin{eqnarray*}
g(x) &=& 2x \log_2(2x) + (1-2x) \log_2(1-2x) - x \log_2 x - (1-x) \log_2 (1-x),\\
c &=& 1/2 - \sqrt{2}/4 \approx 0.1464, \quad \text{and} \\
b &=& - g(c) \approx 0.2716.
\end{eqnarray*}
Suppose $k\geq c(Q_n)$.  The following hold for large enough $n$.
\begin{enumerate}[(i)]
\item If $k \le 2^{n^\alpha}$ for some $\alpha < 1$, then $\capt_k(Q_n) = \Theta( n \log n )$.
\item If $k \le 2^{n(1-b+\varepsilon)}$, then $\Omega(n) = \capt_k(Q_n) = O( n \log n )$.
\item  If $2^{n(1-b+\varepsilon)} < k \le 2^{n(1-\varepsilon)}$, then $\capt_k(Q_n) = \Theta(n)$.
\item  If $k = 2^{n-f(n)}$ with $\log n \ll f(n)=o(n)$, then
\[
\capt_k(Q_n) = \Theta \left( \frac {f(n)}{\log (n / f(n))} \right) = \Theta \left( \frac {n}{\omega \log \omega} \right)
\]
where $\omega=\omega(n)=n/f(n)$ (note that $\omega\to\infty$, so $1 \ll \capt_k(Q_n) = o(n)$).
\item If $k = 2^{n-f(n)}$ with $f(n)=O(\log n)$ (which is equivalent to $k \ge 2^n / n^{O(1)}$), then $\capt_k(Q_n) = O(1)$.
\end{enumerate}
\end{cor}
\begin{proof}
Recall that the capture time with $c(Q_n)$ cops was determined to be $\Theta(n\log n)$ in~\cite{BGKP}; since $\capt_k(G)$ is monotone non-increasing in $k$, this establishes the upper bounds in
parts~(i) and (ii).  The lower bound in part~(i) follows immediately from Theorem~\ref{thm:Qn-lowerbound1}.  The lower bound in part~(iii) (and hence, by monotonicity also in~(ii)) follows from
Theorem~\ref{thm:Qn-lowerbound2}, with $d=\alpha n$ chosen so that $\varepsilon>\log_2((e/\alpha)^{\alpha})+1/n$. For the upper bound in part (iii) note that if $d = cn-2$, then (using Stirling's
formula $x! \sim \sqrt{2\pi x} (x/e)^x$) we have
\begin{align*}
36 \cdot 2^n n \ \frac {(2d+1)_{d+1}}{ (n-d)_{d+1} }
&\sim 36n \cdot 2^{ n } \frac{ \sqrt{2d+1}((2d+1)/e)^{2d+1} } {\sqrt{d} (d/e)^d}
    \frac{\sqrt{n-2d}((n-2d)/e)^{n-2d}}{\sqrt{(n-d)}((n-d)/e)^{n-d}} \\
&= O(1) \cdot 2^n n \frac{ (2cn)^{2cn-1} ((1-2c)n)^{(1-2c)n} }{ (cn)^{cn} ((1-c)n)^{(1-c)n}}\\
&= O(1) \cdot 2^{n} \left(\frac{ (2c)^{2c} (1-2c)^{1-2c} }{c^{c} (1-c)^{1-c}}\right)^n
 = 2^{n(1+ g(c) + o(1))} < k,
\end{align*}
(when $n$ is large enough to make the $o(1)$ term less than $\varepsilon$) and Theorem~\ref{thm:Qn-upperbound2} yields the linear upper bound.
To get the upper bound in part~(iv) we will again use Theorem~\ref{thm:Qn-upperbound2}, this time with $d = 2 n / (\omega \log \omega)$. Since $1 \ll \omega = o(n / \log n)$, we have
\[
36 \cdot 2^n n \ \left( \frac {3d}{n(1-c)} \right)^{d+1} \le 2^{n + O(\log n) + \frac {2n}{\omega \log \omega} \log \left( \frac {8}{\omega \log \omega} \right) } =
2^{n - \frac{(2+o(1)) n \log\log\omega}{\omega} } < 2^{n - f(n)} = k,
\]
and the desired upper bound holds. To get the matching lower bound we will use Theorem~\ref{thm:Qn-lowerbound2} with $d = n / (2 \omega \log \omega)$. This time we need to verify that
\[
\frac 12 \cdot 2^n \left( \frac {d}{en} \right)^{d} = 2^{n-1 + \frac {n}{2 \omega \log \omega} \log \left( \frac {1}{2e \omega \log \omega} \right)} =  2^{n- (1 + o(1)) \frac {n}{2\omega} } > 2^{n - f(n)} = k,
\]
and the desired lower bound holds too.
Finally, part~(v) follows immediately from Theorem~\ref{thm:Qn-upperbound2} with $d$ constant.
\end{proof}

\section{Planar graphs}\label{sec:planar}

We first investigate temporal speed-up on planar graphs if $k = \Omega(\sqrt{n})$.

\begin{thm}\label{planarc}
For any connected planar graph $G$, if $k\geq 12\sqrt{n}$, then $\capt_k(G)\leq 6\rad(G)\log n$.
\end{thm}

\begin{proof}
We use the planar separator theorem of Alon, Seymour, and Thomas~\cite{AST94}: there is a set of at most $2.13\sqrt{n}$ vertices that separate the graph into two sets of size at most $\frac23 n$. Now
\[ k\geq 12\sqrt{n} > 2.13\sqrt{n}+2.13\sqrt{\frac23n}+2.13\sqrt{\left(\frac23\right)^2n}+\dotsb,\]
so place an initial team of cops on a separator of size at most $2.13\sqrt{n}$ (so that each vertex of the separator is covered by exactly one cop) and the rest of the cops on a central vertex.  The robber will place herself on some vertex that is in one of the two subgraphs
separated by the first team of cops.  A second team of cops, of size $2.13\sqrt{\frac23n}$, moves to a separator of that subgraph, while all the other cops remain still.  Repeat this process until
the robber's territory is reduced to nothing.  This requires a number of teams $t$ that satisfies
\[ \left(\frac23\right)^t n < 1, \]
which is true with $t=6\log n$.  Each team of cops takes its position in at most $\rad(G)$ steps.
\end{proof}

Observe that the proof of Theorem~\ref{planarc} works even in a version of the game in which the robber is allowed to move \emph{infinitely fast}; that is, she can move to any vertex in the same
component of the graph minus the cops' vertices.

We can use the version for graphs of genus $g$ of Gilbert, Hutchinson, and Tarjan~\cite{GHT} to obtain the following.
\begin{cor}
Let $G$ be a connected graph with genus $g$, and suppose that $k\geq (19+66\sqrt{g})\sqrt{n}$.  Then $\capt_k(G)\leq 6\rad(G)\log n$.
\end{cor}

For the square grid $P_n\cp P_n$ with $n^2$ vertices and $k=n$ cops, note that the strategy in the proof above takes $O(n\log n)$ time, whereas one can just sweep along the rows to capture in time
$n/2$, and our partitioning scheme for Theorem~\ref{thm:grids} gets it down to $O(\sqrt{n})$ time.

We finish this section by considering the effect of having three cops play on planar graphs.

\begin{thm}\label{thm:3copplanar}
If $G$ is a connected planar graph, then $\capt_3(G)\leq (\diam(G)+1) |V(G)|$.
\end{thm}

The bound in Theorem~\ref{thm:3copplanar} is an improvement over the bound (for any $G$) that for $k\geq c(G)$, $\capt_k(G)\leq n^{c(G)+1}$~\cite{BI}.  It does not, however, improve the bound for
cop-win graphs of $n-3$~\cite{BGHK,gav1}.

\begin{proof}[Proof of Theorem~\ref{thm:3copplanar}]
We follow the proof that 3 cops suffice to catch the robber presented in~\cite{BN} (based on the original proof of~\cite{AF}), with some slight modifications, and give an upper bound on the time the
cop's algorithm given there may take.
 The main observation is that a cop can guard any shortest path $P$ in a subgraph $H$ of $G$ in at most $\diam(G)+|V(P)|/2$ steps: at most $\diam(G)$ steps to reach the central vertex of $P$,
followed by at most $|V(P)|/2$ steps to capture the shadow of the robber on $P$.

The cops maintain a \emph{cop territory} into which the robber cannot enter without being immediately caught; the remainder of the graph is the \emph{unguarded territory}, denoted $H$.
The cops ensure one of the following three cases hold throughout the game (after an initial phase):
\begin{enumerate}[(I)]
\item\label{case:onepath} Some cop is guarding a (nontrivial) shortest path $P$ of $H$, and any path from the robber to the cop territory is through $P$.
\item\label{case:twopaths} Two cops guard $P_1\cup P_2$, where $P_1$ and $P_2$ are internally disjoint paths with the same endpoints, and any path from the robber to the cop territory is through a
    vertex of $P_1\cup P_2$.  ($H$ is either the interior or exterior region of the cycle $P_1\cup P_2$, whichever contains the robber.)
\item\label{case:cutvertex} Some cop guards a single vertex that prevents the robber from leaving $H$.
\end{enumerate}
To begin, we send one cop to guard a shortest path $P$ joining two vertices at maximum distance from each other.  This takes at most $|P|/2$($=\diam(G)/2$) steps and puts us into
Case~\ref{case:onepath} (with $H=G-V(P)$).  The cops' strategy now repeatedly reduces the unguarded territory $H$ while ensuring one of the three cases holds at all times.  We claim that in the $i$th
phase, we take at most $\diam(G)+k_i$ steps and reduce the unguarded territory by at least $k_i$ vertices for some $k_i$.  This will imply that the robber is caught when $\sum_i k_i = n$, which will
take at most $\sum_i (\diam(G)+k_i) \leq n\diam(G)+n$ steps.

\smallskip
\noindent\textbf{Case~\ref{case:onepath}:}  Let $Y$ denote the component of $H-V(P)$ containing the robber, and let $C_1$ be the cop guarding $P$.

If there is a unique $v\in V(P)$ with a neighbour in $Y$, then $C_1$ prevents the robber from reaching $v$, so we are actually in Case~\ref{case:cutvertex}.  This is not counted as a phase, as it
does not require additional time.

Otherwise, let $v_1$ and $v_2$ be the first and last vertices of $P$ that have neighbours in $Y$, and let $u_1$ and $u_2$ be such neighbours (respectively).  Let $P_2$ be a shortest $(u_1,u_2)$-path
in $Y$, and move $C_2$ to guard $P_2$.  This takes at most $\diam(G)+|V(P_2)|/2$ steps. Let $P_1$ be the portion of $P$ from $v_1$ to $v_2$. The robber's territory is now $Y$ restricted to either the
inside or outside of the cycle $P_1\cup P_2\cup\{v_1u_1, v_2u_2\}$.  Thus we are in Case~\ref{case:twopaths} and we have reduced $H$ by at least $|V(P_2)|$.

\smallskip\noindent\textbf{Case~\ref{case:twopaths}:}  Let $C_1$ and $C_2$ be the cops guarding $P_1$ and $P_2$, respectively.  Let $X=P_1\cup P_2$, and let $Y$ be the component of $H$ containing the robber.

If there is a unique $v\in V(X)$ with a neighbour in $Y$, then one of $C_1$ and $C_2$ prevents the robber from escaping $Y$ through $v$, so we are actually in Case~\ref{case:cutvertex}.  Again we do
not count this as a phase.

If each of $P_1$, $P_2$ have exactly one vertex (say $v_1$, $v_2$ respectively) with a neighbour in $Y$, then let $K=G[V(Y)\cup\{v_1,v_2\}]-\{v_1v_2\}$.  Let $P$ be a shortest $(v_1,v_2)$-path in
$K$; note that $P$ contains at least one vertex of $Y$.  Send $C_3$ to guard $P$.  The robber cannot reach $v_1$ or $v_2$ without being caught by $C_1$ or $C_2$, so this takes at most
$\diam(G)+(|V(P)|-2)/2$ steps.  Once $C_3$ is in place, note that the robber cannot safely reach $P_1$ or $P_2$, so $C_1$ and $C_2$ are free again to move.  We are now in Case~\ref{case:onepath} and
have reduced $H$ by at least $|V(P)|-2$ vertices.

Finally, suppose $P_1$ has at least two vertices with neighbours in $Y$; let $v_1,v_2$ be the first and last such vertices of $P_1$.  Let $u_1,u_2$ be neighbours of $v_1,v_2$ (respectively) in $Y$.
Let $P$ be a shortest $(u_1,u_2)$-path in $Y$, and send $C_3$ to guard $P$.  Let $Q$ denote the subpath of $P_1$ from $v_1$ to $v_2$.  If the robber is in the region bounded by $Q\cup
P\cup\{v_1u_1,v_2u_2\}$, then we are in Case~\ref{case:twopaths} with $C_2$ free to move.  If instead the robber is in the region bounded by $(P_1-Q)\cup P_2\cup P\cup\{v_1u_1,v_2u_2\}$, then let
$P'=(P_1-Q)\cup P\cup\{v_1u_1,v_2u_2\}$; note that $P'$ is a shortest $(v_1,v_2)$-path in the region bounded by $P'\cup P_2$, so $C_3$ may actually guard all of $P'$.  Hence we are in
Case~\ref{case:twopaths} with $C_1$ free to move. In either case we have taken at most $\diam(G)+|V(P)|/2$ steps and have reduced $H$ by $|V(P)|$ vertices.

\smallskip\noindent\textbf{Case~\ref{case:cutvertex}:} Let $Y'$ be the component of $H$ containing the robber, and let $Y=G[Y'\cup\{v\}]$.  Let $u$ be a vertex of maximum distance (in $Y$) from $v$, and let $P$ be a shortest $(u,v)$-path in $Y$.  Move a free cop to guard $P-v$; once he is in place, he guards all of $P$.  We are in Case~\ref{case:onepath}, have taken at most $\diam(G)+(|V(P)|-1)/2$ turns, and have reduced $H$ by $|V(P)|-1$ vertices.
\end{proof}

\section{Binomial Random Graphs}

The binomial random graph $\mathcal{G}(n,p)$ is defined as a random graph with vertex set $[n]=\{1,2,\dots, n\}$ in which a pair of vertices appears as an edge with probability $p$, independently for
each such a pair. As typical in random graph theory, we consider only asymptotic properties of $\mathcal{G}(n,p)$ as $n\rightarrow \infty$, where $p=p(n)$ may and usually does depend on $n$. We say
that an event in a probability space holds \emph{asymptotically almost surely} (\emph{a.a.s}.) if its probability tends to one as $n$ goes to infinity.

We first briefly describe some known results on the cop number of $\mathcal{G}(n,p)$. The first and third author along with Wang investigated such games in $\mathcal{G}(n,p)$ random graphs, and their
generalizations used to model complex networks with a power-law degree distribution (see~\cite{bpw}). From their results it follows that if $2 \log n / \sqrt{n} \le p < 1-\epsilon$ for some
$\epsilon>0$, then a.a.s.\ $ c(\mathcal{G}(n,p))= \Theta(\log n/p). $ Bollob\'as, Kun and Leader~\cite{bkl} showed that for $p(n) \ge 2.1 \log n /n$, then a.a.s.
$$
\frac{1}{(pn)^2}n^{ 1/2 - 9/(2\log\log (pn))  }  \le c(\mathcal{G}(n,p))\le 160000\sqrt n \log n\,.
$$
From these results, if $np \ge 2.1 \log n$ and either $np=n^{o(1)}$ or $np=n^{1/2+o(1)}$, then a.a.s.\ $c(\mathcal{G}(n,p))= n^{1/2+o(1)}$. Somewhat surprisingly, between these values
$c(\mathcal{G}(n,p))$ was shown by \L{}uczak and the third author~\cite{Expansion3} to have more complicated behaviour. It follows that a.a.s.\ $\log_n  c(\mathcal{G}(n,n^{x-1}))$ is asymptotic to
the function $f(x)$ shown in Figure~\ref{fig1}.
\begin{figure}
\begin{center}
\includegraphics[width=4in]{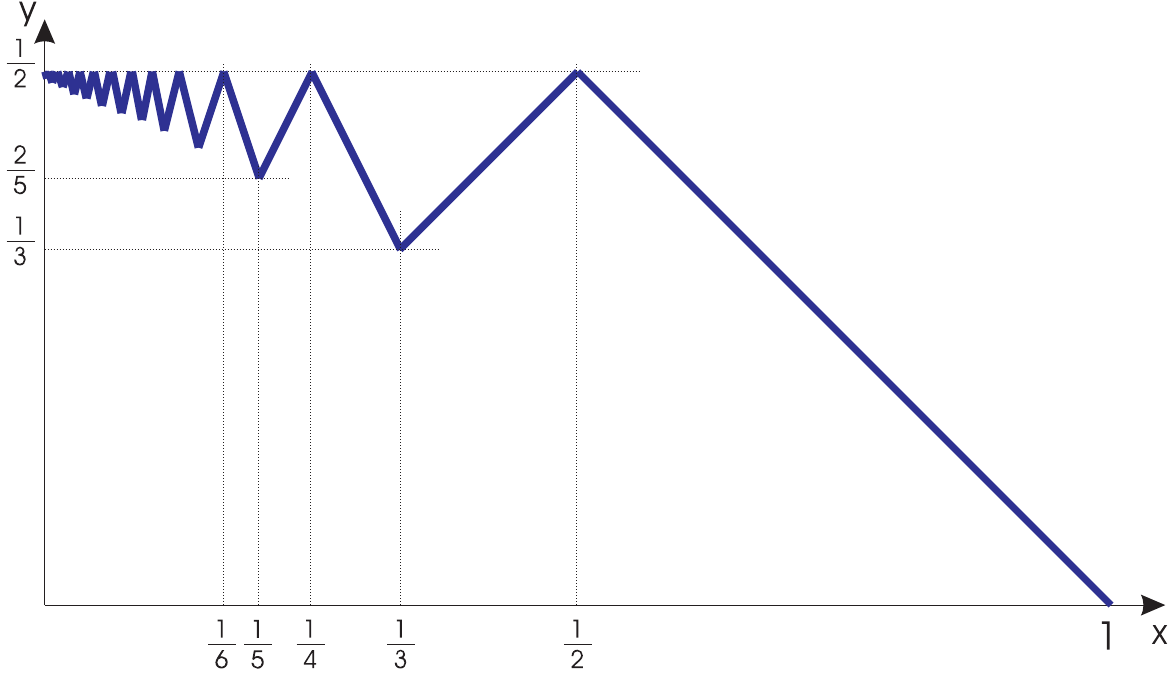}
\end{center}
\caption{The ``zigzag'' function $f$.}\label{fig1}
\end{figure}

Using ideas from~\cite{Expansion3,Expansion2}, we may obtain bounds for the capture time of binomial random graphs. For simplicity, we restrict ourselves to dense random graphs ($d=p(n-1)\ge \log^3
n$) and a large number of cops ($k = k(n) \ge C \sqrt{n \log n}$). Further, we present a sketch of the proof only. Adjusting the argument to sparser graphs (based on the more sophisticated argument
in~\cite{Expansion2} for the sparse case) should be straightforward, but the upper bound will not match the lower bound. Similarly, adjusting the argument to fewer cops is possible but definitely not
all the way to the cop number. Investigating the capture time for $k=c(\mathcal{G}(n,p))$, even for very dense random graphs (say, for $p=1/2$) appears to be a challenging problem. We do not even
know the exact value of the cop number there! For further background, the reader is directed to~\cite{Expansion3, Expansion1,Expansion2}.

\begin{thm}\label{thm:dense_case}
Suppose that $d=p(n-1)\ge \log^3 n$ and $C \sqrt{n \log n} \le k = k(n) < n$ for some sufficiently large constant $C$.
Finally, let $r=r(d,k)$ be the smallest positive integer such that $d^{r+1} \ge C n \log n / k$. Let $G=(V,E) \in \mathcal{G}(n,p)$. Then a.a.s.\
$$
\capt_k(G) = \Theta(r).
$$
\end{thm}

\begin{proof}
As referenced above, we sketch the proof only. First, let us mention that a.a.s.\ $\mathcal{G}(n,p)$ is a good expander. Let $N(v,j)$ be the set of vertices at distance at most $j$ from vertex $v$.
One can show that a.a.s.\ for any vertex $v$ and every $j$ such that $d^j = o(n)$, $N(v,j) = (1+o(1)) d^j$. Moreover, it is well known that any graph $G$ with minimum degree $\delta=\delta(G)>2$ has
a dominating set of size $O(n \log \delta / \delta)$. Hence, we may assume that $d < \sqrt{n \log n}$ as for denser graphs we immediately get that a.a.s.\ $\capt_k(G) = \Theta(1)$ for any $C \sqrt{n
\log n} \le k < n$. Finally, we may assume that $d < C n \log n / k$. Indeed, if $k$ is too large so that $d \ge C n \log n / k$, then the result for $k' = C \sqrt{n \log n}$ implies that a.a.s.\
$\capt_k(G) \le \capt_{k'}(G)= \Theta(1)$.

We place $k$ cops at random as in the proof of Theorem~\ref{thm:Qn-upperbound2}. The robber appears at some vertex $v \in V$. Note that it follows from the definition of $d$ that
$$
d^r< \frac {Cn\log n}{k} \le \sqrt{n \log n} \le \frac {k}{C}.
$$
The main difficulty is to show that with probability $1-o(n^{-1})$, it is possible to assign distinct cops to all vertices $u$ in $N(v,r) \setminus N(v,r-1)$ such that a cop assigned to $u$ is within
distance $(r+1)$ of $u$. (Note that here, the probability refers to the randomness in distributing the cops; the random graph is fixed.) If this can be done, then after the robber appears these cops
can begin moving straight to their assigned destinations in $N(v,r) \setminus N(v,r-1)$. Since the first move belongs to the cops, they have $r+1$ steps to do so, after which the robber must still be
inside $N(v,r)$, while $N(v,r)\setminus N(v,r-1)$ is fully occupied by cops. Then in at most $r$ additional steps, the cops can ``tighten the net'' around $v$ and capture the robber. Hence, the cops
will win after at most $2r+1$ steps with probability $1-o(n^{-1})$, for each possible starting vertex $v \in V$. Hence, this strategy gives a win for the cops a.a.s.

We will use Hall's theorem for bipartite graphs to show that the desired assignment exists. We need to verify that, for any set $S\subseteq N(v,r) \setminus N(v,r-1)$, there are at least $|S|$ cops
lying on vertices within distance $r+1$ from some vertex in $S$. One thing to make sure of is that Hall's condition holds for $S=N(v,r) \setminus N(v,r-1)$. It follows from expansion properties that
$|N(v,r)| < 2 d^r \le 2k/C$, so there are enough cops to achieve this goal. The main bottleneck is to satisfy the condition for sets with $|S|=1$. Since for any vertex $u$, the expected number of
cops in $N(u,r+1)$ is asymptotic to $k d^{r+1} / n \ge C \log n$, the condition holds provided that $C$ is large enough (see~\cite{Expansion3} or~\cite{Expansion2} for more details).

In order to get the lower bound, we need to use expansion properties again. It is possible to show that a.a.s.\ for any starting position of $k$ cops, the number of vertices at distance at most $r-1$
from them is asymptotic to $k d^{r-1}$, since $k d^{r-1} < C n \log n / d = o(n)$. The robber can start at distance at least $r$ from any cop and wait there.
\end{proof}

\end{document}